\documentclass[12pt]{amsart}

\usepackage{tikz}
\usepackage{amsmath}
\usepackage{amsthm}
\usepackage{hyperref}
\usepackage[margin=1.0in]{geometry}

\numberwithin{equation}{section}

\newtheorem{prop}{Proposition}
\newtheorem{lemma}[prop]{Lemma}

\newtheorem{thm}[prop]{Theorem}
\newtheorem{cor}[prop]{Corollary}

\numberwithin{prop}{section}

\theoremstyle{definition}
\newtheorem{defn}[prop]{Definition}

\newtheorem{rmk}[prop]{Remark}

\newcommand{\del}{\partial}
\newcommand{\delb}{\bar{\partial}}\newcommand{\dt}{\frac{\partial}{\partial t}}
\newcommand{\brs}[1]{\left| #1 \right|}

\renewcommand{\gg}{\gamma}
\newcommand{\gD}{\Delta}

\newcommand{\gl}{\lambda}

\newcommand{\gw}{\omega}
\newcommand{\ga}{\alpha}
\newcommand{\gb}{\beta}
\renewcommand{\ge}{\epsilon}
\newcommand{\N}{\nabla}
\newcommand{\FF}{\mathcal F}

\newcommand{\CC}{\mathcal C}

\newcommand{\til}[1]{\widetilde{#1}}

\renewcommand{\bar}[1]{\overline{#1}}

\newcommand{\HH}{\mathcal H}

\newcommand{\IP}[1]{\left<#1 \right>}

\DeclareMathOperator{\Vol}{Vol}

\begin{document}

\title[The consistency and convergence of K-energy minimizing movements]{The
consistency and convergence of K-energy minimizing movements}

\begin{abstract} We
show that $K$-energy minimizing movements agree with smooth solutions to Calabi
flow as long as the latter exist.  As corollaries we conclude that in a general
K\"ahler class  long time solutions of Calabi flow minimize both $K$-energy and
Calabi energy.  Lastly, by applying convergence results from the theory of
minimizing movements, these results imply that
long time solutions to Calabi flow converge in the weak distance topology to
minimizers of
the $K$-energy functional on the metric completion of the space of K\"ahler
metrics, assuming one exists.  
\end{abstract}

\date{\today}

\author{Jeffrey Streets}
\address{Rowland Hall\\
         University of California, Irvine\\
         Irvine, CA 92617}
\email{\href{mailto:jstreets@uci.edu}{jstreets@uci.edu}}

\maketitle

\section{Introduction}

Let $(M^{2n}, \gw, J)$ be a compact K\"ahler manifold. Let $\HH = \{\phi \in
C^{\infty}(M) | \gw + \sqrt{-1} \del\delb \phi > 0 \}$ denote the space of
K\"ahler potentials, and given $\phi \in \HH$ let $\gw_{\phi} := \gw + \sqrt{-1}
\del
\delb \phi > 0$, and let $s_{\phi}$ denote the scalar curvature of the metric
$\gw_{\phi}$.  Furthermore, let $V = \Vol(\gw_{\phi})$, which is fixed for all
$\phi$, and set $\bar{s} = \frac{1}{V}
\int_M
s_{\phi} \gw_{\phi}^n$, which is also fixed for any $\phi$.  A one-parameter
family
of K\"ahler potentials $\phi_t$ is a solution of \emph{Calabi flow}
if
\begin{align} \label{Cflow}
\dt \phi =&\ s_{\phi} - \bar{s}.
\end{align}
\noindent This flow was introduced by Calabi in his seminal paper
\cite{CalabiExtremal} on extremal K\"ahler metrics.  In \cite{SMM} the author
proved the general long time existence of certain weak solutions to Calabi flow,
which we refer to herein as \emph{$K$-energy minimizing movements}, (\emph{KEMM}
for short).  In this paper we derive some further properties of these weak
solutions focusing on their relationship to smooth Calabi flows, and we discuss
the implications for Calabi flow.

The first main result is a consistency theorem relating KEMM and smooth
solutions to Calabi flow.  In particular we show that if the Calabi flow with
some initial condition $\phi_0$ exists smoothly on some time interval, then the
KEMM with initial condition $\phi_0$ agrees with the solution to Calabi flow on
that interval.

\begin{thm} \label{realizationthm} Let $(M^{2n}, \gw, J)$ be a compact K\"ahler
manifold.  Given $\phi_0 \in \HH$, suppose the Calabi flow with initial
condition $\phi_0$, call it $\phi_t$, exists smoothly on $[0, T)$.  Let
$\til{\phi}_t$ denote the $K$-energy minimizing movement with initial condition
$\phi_0$.  Then $\phi_t = \til{\phi}_t$ for all $t \in [0, T)$.
\end{thm}

\noindent This theorem has several direct consequences for smooth solutions to
Calabi flow coming from the theory of minimizing movements.  The first example
is a bound on the growth of distance along a solution to Calabi flow.

\begin{cor} \label{distancecor} Let $(M^{2n}, \gw, J)$ be a compact K\"ahler
manifold.  Given
$\phi_0 \in \HH$, if $\phi_t$ denotes a smooth Calabi flow with initial
condition
$\phi_0$ on $[0, T]$, then there is a constant $C = C(\phi_0, \gw, T)$ such that
for all $0
\leq s \leq t < T$ one has
\begin{align*}
d(\phi_s, \phi_t) \leq&\ C (t-s)^{\tfrac{1}{2}}.
\end{align*}
\end{cor}

As the constant in the above corollary depends on $T$, it is difficult to apply
this effectively for solutions existing on an infinite time interval.  We
provide one further corollary in this direction.

\begin{cor} \label{distancecor2} Let $(M^{2n}, \gw, J)$ be a compact K\"ahler
manifold, and suppose the $K$-energy in $[\gw]$ is bounded below.  Given 
$\phi_0 \in \HH$, if $\phi_t$ denotes a smooth Calabi flow with initial
condition
$\phi_0$ on $[0, T]$, then there is a constant $C = C(\nu(\phi_0), [\gw])$ such
that
for all $0
\leq s \leq t < T$ one has
\begin{align*}
d(\phi_s, \phi_t) \leq&\ C (t-s)^{\tfrac{1}{2}}.
\end{align*}
\end{cor}

\begin{rmk} A direct estimate of $d(\phi_0, \phi_t)$ along a solution to Calabi
flow using the variation of
length and the a priori bound on Calabi energy yields a linear growth rate for
large times.  Thus Corollaries \ref{distancecor} and \ref{distancecor2} yield a
nontrivial improvement
of this estimate.
\end{rmk}

Another consequence of Theorem \ref{realizationthm} shows that long time
solutions of Calabi flow always realize the infimum of $K$-energy, denoted
$\nu$,
and Calabi
energy, denoted $\CC$.  A key component of the proof is a type of ``evolutionary
variational
inequality'' for the $K$-energy along a solution to Calabi flow.  As this is of
some independent interest we include the statement here.

\begin{cor} \label{evicor} Let $(M^{2n}, \gw, J)$ be a compact K\"ahler
manifold.  Given $\phi_t$ a solution to Calabi flow on $[0, T]$ and $\psi \in
{\HH}$, then one has
\begin{align*}
 d^2(\phi_{t+s}, \psi) \leq d^2(\phi_t, \psi) - 2s \left( \nu(\phi_{t+s}) -
\nu(\psi) \right)
\end{align*}
for all $t, s \geq 0, t+s \leq T$.
\end{cor}

\begin{cor} \label{minimizingcor} Let $(M^{2n}, \gw, J)$ be a compact K\"ahler
manifold.  Given $\bar{\phi_0} \in \bar{\HH}$ the KEMM with initial condition
$\bar{\phi}_0$ satisfies
\begin{align*}
 \lim_{t \to \infty} \bar{\nu}(\bar{\phi}_t) = \inf_{\phi \in {\HH}}
{\nu}({\phi}) .
\end{align*}
Furthermore, given $\phi_t \in \HH$ a solution of
Calabi flow on $[0,\infty)$, one has
\begin{align*}
 \lim_{t \to \infty} \nu(\phi_t) = \inf_{\phi \in \HH} \nu(\phi).
\end{align*}
\end{cor}

\begin{cor} \label{minimizingcor2} Let $(M^{2n}, \gw, J)$ be a compact K\"ahler
manifold.  Given $\phi_t$ a solution of Calabi flow on $[0,\infty)$, one has
\begin{align*}
\lim_{t \to \infty} \CC(\phi_t) = \inf_{\phi \in \HH} \CC(\phi).
\end{align*}
\end{cor}

Next we establish convergence of $K$-energy minimizing movements to $K$-energy
minimizers in the weak distance topology, assuming such a minimizer
exists.  This is a direct consequence of a general theorem on weak convergence
of minimizing
movements to minimizers proved by Ba\u{c}\'ak (\cite{Bacak} Theorem 1.5).  As
the proof of this theorem is
ultimately spread through many papers, we include a self-contained exposition in
\S \ref{convsec} for convenience.

\begin{thm} \label{convthm} Let $(M^{2n}, \gw, J)$ be compact K\"ahler manifold.
 Suppose $\bar{\phi} \in \bar{\HH}$ is a minimizer for $\bar{\nu}$.  Given
$\bar{\phi}_0 \in \bar{\HH}$,
the KEMM with initial condition $\bar{\phi}_0$ exists for all time and converges
weakly to a minimizer for $\bar{\nu}$ in the distance topology.
\end{thm}

\noindent Combining Theorems \ref{realizationthm} and \ref{convthm}, and using
the result of Chen-Tian \cite{ChenTian} that constant scalar curvature metrics
are minimizers of $K$-energy yields the
obvious corollary:

\begin{cor} \label{convcor} Let $(M^{2n}, \gw, J)$ be a compact K\"ahler
manifold and suppose
$\phi_{\infty} \in \HH$ satisfies $s_{\phi_{\infty}} \equiv c$.  Then any
solution to Calabi flow which exists smoothly on $[0, \infty)$ converges weakly
to a minimizer for $\bar{\nu}$ in the distance topology.
\end{cor}

\begin{rmk} Corollary \ref{convcor} represents an affirmative qualitative answer
to a question implicit in \cite{DonaldsonConj}.  In particular, in
\cite{DonaldsonConj} Donaldson proposes four possibilities for the convergence
behavior of long time solutions of Calabi flow.  The simplest case is that of
convergence to a constant scalar curvature metric, assuming one exists. 
Corollary \ref{convcor} schematically takes this form, 
but the statement hides a
subtlety that prevents it from being a complete answer to this conjecture.  In
particular, while it follows from the work of Chen-Tian \cite{ChenTian} that
constant scalar curvature metrics represent the only minima of $K$-energy, it
does not immediately follow that the minimizers for $\bar{\nu}$ are all in
$\HH$, and so are minimizers for $\nu$.  Given the convexity properties of $\nu$
though this seems likely to be true.  What is missing is a kind of ``effective
uniqueness'' statement which says that any sequence of points in $\HH$ realizing
the infimum of $K$-energy converges in the weak distance topology to a constant
scalar curvature metric.  A result of this kind together with the corresponding
convergence statement for Calabi flow is established for the (anti)canonical
K\"ahler class on a K\"ahler-Einstein manifold by Berman \cite{Berman}.
\end{rmk}

Here is an outline of the rest of the paper.  In \S \ref{MMreview} we review the
construction of $K$-energy minimizing movements and discuss some of their
properties.  Then in \S \ref{realthmpf} we give the proof of Theorem
\ref{realizationthm} and the related corollaries.  We end in \S \ref{convsec}
with the proof of Theorem
\ref{convthm} and Corollary \ref{convcor}.

\vskip 0.2in

\textbf{Acknowledgments:} The author would like to thank Miroslav Ba\u{c}\'ak
for
informing the author of his work, and also Robert Berman and Weiyong He for
interesting
discussions on Calabi flow.

\section{Review of K-energy minimizing movements} \label{MMreview}

The method of minimizing movements consists of employing an implicit Euler
scheme to generate gradient lines of functionals on metric spaces.  A
foundational work in this direction is the paper of Mayer \cite{Mayer}, proving
a general existence result in the case of convex functionals
on metric spaces with nonpositive curvature in the sense of Alexandrov.  In
\cite{SMM} we showed that one can use this theorem to construct minimizing
movement solutions to Calabi flow which exist for all time.  In this section we
give a very brief review of this construction,
and then record some further aspects of the theory of minimizing movement
established in \cite{Mayer}.

\subsection{The space of K\"ahler metrics}

Let $(M^{2n}, \gw, J)$ be a compact K\"ahler manifold.  As in the introduction,
we denote the space of
K\"ahler metrics in $[\gw]$, thought of as a space of K\"ahler potentials, by
$\HH$:
\begin{align*}
 \HH = \{\phi \in C^{\infty}(M) | \gw + \sqrt{-1} \del\delb \phi > 0 \}.
\end{align*}
This space is an infinite dimensional manifold modeled locally on
$C^{\infty}(M)$.  In particular, formally one observes that for every $\phi \in
\HH$, $T_{\phi} \HH \cong C^{\infty}(M)$, and we can define a Riemannian metric,
called the Mabuchi-Semmes-Donaldson metric (\cite{Mabuchi}, \cite{Semmes},
\cite{Donaldson}) by
\begin{align*}
 \IP{\ga,\gb}_{\phi} := \int_M \ga\gb \gw_{\phi}^n.
\end{align*}
As shown in (\cite{Mabuchi}, \cite{Semmes}, \cite{Donaldson}), this metric has
formally nonpositive curvature.  One can adapt the usual definition of the
length of a curve to this situation, and it was shown in the work of Chen
\cite{ChenSOKM} that the space $\HH$ is convex by $C^{1,1}$ geodesics, and
moreover the distance function induced by the length functional does indeed
induce a metric space structure on $\HH$, which we denote by $d$.  Given this
setup, we let $(\bar{\HH}, d)$ denote the metric space completion of $(\HH, d)$.

\subsection{Long time existence of K-energy minimizing movements}

Recall that the $K$-energy functional can be defined on $\HH$ by
\begin{align*}
\nu(\phi) =&\ - \int_0^1 \int_M \left( s(\gw_{\phi}) - \bar{s} \right)
\dot{\phi} \gw_{\phi}^n dt.
\end{align*}
Where $\phi_t : [0,1] \to \HH$ is any smooth map such that $\phi_0 = 0$ and
$\phi_1 = \phi$.  We aim to construct minimizing movements of $\nu$ in
$\bar{\HH}$, so first we must extend the domain of $\nu$ to $\bar{\HH}$.

\begin{defn} \label{lscextdefn} Let $(X, d)$ be a metric space and $f : X \to
\mathbb R$ a lower
semicontinuous function.  If $(\bar{X}, d)$ denotes the completion of $(X,
d)$, we define the \emph{lower semicontinuous extension of $f$} by
\begin{align} \label{lscext}
\bar{f}(x) := \begin{cases}
f(x) & x \in X\\
\displaystyle \liminf_{x_n \to x} f(x_n) & x \in \bar{X} \backslash X.
\end{cases}
\end{align}
One can easily show that $\bar{f}$ is
indeed lower semicontinuous (cf. \cite{SMM} Lemma 5.10).
\end{defn}

\begin{defn} Let $(M^{2n}, \gw, J)$ be a compact K\"ahler manifold.  Observe
that by Theorem \ref{Kenergydecay} the function $\nu$ is lower semicontinuous on
$\HH$.  Thus we set
\begin{align*}
\bar{\nu} : \bar{\HH} \to \mathbb R
\end{align*}
to be the lower semicontinuous extension of $\nu : \HH \to \mathbb R$ in the
sense of Definition \ref{lscextdefn}.
\end{defn}

\begin{defn} \label{discdef} Let $(M^{2n}, \gw, J)$ be a compact K\"ahler
manifold.  Fix ${\phi}
\in \bar{\HH}$ and $\tau > 0$.  Let
\begin{align} \label{Fdef}
\mathcal F_{{\phi}, \tau} ({\psi}) =&\ \frac{{d}^2({\phi}, {\psi})}{2 \tau}
+
\bar{\nu}({\psi}).
\end{align}
Furthermore, set
\begin{align} \label{mudef}
\mu_{\phi,\tau} := \inf_{\psi \in \HH} \FF_{\phi,\tau}(\psi).
\end{align}
The quantity $\mu$ is sometimes referred to as a \emph{Moreau-Yosida
approximation} of the given functional, in this case $\nu$.  Finally, we define
the \emph{resolvent operator}
\begin{align*}
W_{\tau} : \bar{\HH} \to \bar{\HH}
\end{align*}
by the property
\begin{align*}
\FF_{\phi,\tau}(W_{\tau}(\phi)) = \mu_{\phi,\tau}.
\end{align*}
The fact that there exists a unique minimizer for $\FF_{\phi,\tau}$ and so the
map $W_{\tau}$ is shown as part of the proof of (\cite{Mayer} Theorem 1.13).
\end{defn}

\noindent One interprets the resolvent operator formally as $W_{\tau} = (I +
\tau \N \bar{\nu})^{-1}$, and as such one expects that the formal gradient line
for $\bar{\nu}$ with initial condition $\phi$ can be constructed by $\phi_t =
\displaystyle \lim_{n \to \infty} W_{\frac{t}{n}}^n(\phi_0)$.  This formal
picture was
justified in the case of convex functions on metric spaces of nonpositive
curvature in the work of Mayer (\cite{Mayer} Theorem 1.13).  In \cite{SMM} we
applied this theorem to assert the long time existence of formal gradient lines
of $\bar{\nu}$ with arbitrary initial data.

\begin{thm} \label{MMex} (\cite{SMM} Theorem 1.3) Given $\phi_0 \in \bar{\HH}$,
there
exists a continuous map
\begin{align*}
 \phi_t : [0,\infty) \to \bar{\HH}
\end{align*}
such that for all $t > 0$ one has
\begin{align*}
 \phi_t = \lim_{n \to \infty} W_{\frac{t}{n}}^n(\phi_0),
\end{align*}
and moreover
\begin{align*}
 \lim_{t \to 0} \phi_t = \phi_0.
\end{align*}
\end{thm}

\subsection{Further aspects of Mayer's Theorem}

In this subsection we record a number of further properties of the minimizing
movements constructed in Mayer's Theorem.  These generally speaking are
generalizations of certain properties obviously satisfied for smooth gradient
flows. 
While the statements we reference in \cite{Mayer} apply to the general setup of
that paper, we have specialized the statements to our situation for convenience.
 To begin we record a definition of the lower slope of $\bar{\nu}$ at a point.

\begin{defn} Given $\phi_0 \in \bar{\HH}$ satisfying $\bar{\nu} ({\phi_0}) <
\infty$, let
 \begin{align*}
  \brs{\N_- \bar{\nu}}(\phi_0) = \max \left\{ \limsup_{\phi \to \phi_0}
\frac{\bar{\nu}(\phi_0) - \bar{\nu}(\phi)}{d(\phi_0, \phi)}, 0 \right\}
 \end{align*}
\end{defn}

One should think of this quantity roughly speaking as the norm of the gradient
of $\bar{\nu}$, which corresponds to the speed of a minimizing movement.  In
particular, we have

\begin{thm} \label{mayer2} (\cite{Mayer} Theorem 2.17) Let $\phi_t$ be a KEMM. 
Then for all $t \geq 0$ one has
 \begin{align*}
  \lim_{s \to 0^+} \frac{d(\phi_{t+s}, \phi_t)}{s} = \brs{\N_- \bar{\nu}}
(\phi_t).
 \end{align*}
\end{thm}

Part of the proof of this theorem is a lemma we require asserting the
intuitively clear statement that if the infinitesimal variation of distance
along a minimizing movement is zero at some point, then the flow is constant
from that point on.  

\begin{lemma} \label{mayer3} (\cite{Mayer} Lemma 2.15) Let $\phi_t$ be a KEMM. 
If
\begin{align*}
\lim_{s \to 0^+} \frac{d(\phi_{t_0 + s}, \phi_{t_0})}{s} = 0,
\end{align*}
then $\phi_{t} = \phi_{t_0}$ for all $t \geq t_0$.
\end{lemma}

We also later require the again intuitively clear point that the lower gradient
is itself nonincreasing upon taking resolvents due to convexity of $\nu$.

\begin{lemma} \label{mayer4} (\cite{Mayer} Lemma 2.23) Given $\phi \in
\bar{\HH}$ and $\tau > 0$,
\begin{align*}
\brs{\N_- \bar{\nu}}(W_{\tau}(\phi)) \leq&\ \brs{\N_- \bar{\nu}}(\phi).
\end{align*}
Moreover, if $\phi_t$ denotes a KEMM and $t \geq s \geq 0$, then
\begin{align*}
 \brs{\N_- \bar{\nu}}(\phi_t) \leq \brs{\N_-\bar{\nu}}(\phi_s).
\end{align*}
\end{lemma}

It is more
difficult to define the \emph{direction} of the minimizing movement.   The next
theorem asserts roughly speaking that there is a unique steepest
direction associated to minimizing movement solutions.

\begin{thm} \label{uniquedirection} (\cite{Mayer} Theorem 2.16) Let $\phi_t$ be
a KEMM.  Assume $0 < \brs{\N_- \bar{\nu}}(\phi_{t_0})$ and let $\phi^i \to
\phi_{t_0}$ be any sequence of points satisfying
\begin{align} \label{ud10}
 \lim_{i \to \infty} \frac{\bar{\nu}(\phi_{t_0}) -
\bar{\nu}(\phi^i)}{d(\phi_{t_0}, \phi^i)} = \brs{\N_- \bar{\nu}}(\phi_{t_0}).
\end{align}
Then there exists a sequence $s_i \to 0^+$ such that
\begin{align*}
 \lim_{i \to \infty} \frac{d(\phi^i, \phi_{t_0 + s_i})}{d(\phi^i, \phi_{t_0})} =
0.
\end{align*}
\end{thm}

\begin{rmk} \label{uniquedirectionrmk} In the sequel we will use the fact that
the values $s_i$ above are chosen so that
\begin{align*}
 d(\phi_{t_0}, \phi_{t_0 + s_i}) = d(\phi_{t_0}, \phi^i).
\end{align*}
\end{rmk}

Furthermore, it was shown by Calabi-Chen \cite{CalabiChen} that the distance
between pairs of points is nonincreasing under Calabi flow.  This fact can be
generalized to $K$-energy minimizing movements.  Note that we have cited
\cite{SMM}, although in some sense as a corollary to Theorem \ref{MMex} it is
implicit in \cite{Mayer}.

\begin{thm} (\cite{SMM} Theorem 1.4)\label{distnonincr} Let $(M^{2n}, \gw,
J)$ be a compact K\"ahler manifold.  If $\phi_t, \psi_t$ are $K$-energy
minimizing movements, then for
all $t \geq 0$ one has
\begin{align*}
 d(\phi_t, \psi_t) \leq d(\phi_0, \psi_0).
\end{align*}
\end{thm}

Lastly, for the proof of Corollary \ref{minimizingcor2} we require a technical
lemma of Mayer whose proof we reproduce in our simpler case here.

\begin{lemma} \label{minimizinglemma} (\cite{Mayer} Lemma 2.8) Given $\phi_t \in
\bar{\HH}$ a KEMM and
$\psi \in \bar{\HH}$, then for any $s,t \geq 0$ one has
\begin{align*}
 d^2(\phi_{t+s}, \psi) \leq d^2(\phi_t, \psi) - 2s \left( \bar{\nu}(\phi_{t+s})
-
\bar{\nu}(\psi) \right).
\end{align*}
\begin{proof} This is the statement of \cite{Mayer} Lemma 2.8 in the special
case $S = 0$.  By the semigroup properties it suffices to show the statement for
$t = 0$.  Recall that $W_{\tau}$ denotes the resolvent operator.  Let
$\eta_{\gl} : [0,1] \to \bar{\HH}$ denote the unique geodesic connecting
$W_{\tau}(\phi_0)$ to $\psi$.
\begin{align*}
 \FF_{\phi_0, \tau}(W_{\tau}(\phi_0)) =&\ \frac{d^2(W_{\tau}(\phi_0), \phi_0)}{2
\tau} + \bar{\nu}(W_{\tau}(\phi_0))\\
 \leq&\ \frac{d^2(\eta_{\gl},\phi_0)}{2\tau} + \bar{\nu}(\eta_{\gl})\\
 \leq&\ (1 - \gl) \bar{\nu}(W_{\tau}(\phi_0)) + \gl \bar{\nu}(\psi)\\
 &\ + \frac{1}{2\tau} \left( (1-\gl) d^2(W_{\tau}(\phi_0), \phi_0) + \gl
d^2(\phi_0, \psi) - \gl(1-\gl) d^2(W_{\tau}(\phi_0), \psi) \right)\\
=&\ \FF_{\phi_0,\tau}(W_{\tau}(\phi_0)) + \gl \left( \bar{\nu}(\psi) -
\bar{\nu}(W_{\tau}(\phi_0)) + \frac{d^2(\phi_0, \psi)}{2 \tau} -
\frac{d^2(W_{\tau}(\phi_0),\phi_0)}{2\tau} \right)\\
&\ -\gl(1-\gl) \frac{d^2(W_{\tau}(\phi_0),\psi)}{2\tau}.
 \end{align*}
By subtracting $\FF_{\phi_0,\tau}(W_{\tau}(\phi_0))$ from both sides and diving
by $\gl$ we obtain
\begin{align*}
 d^2(W_{\tau}(\phi_0), \psi) \leq d^2(\phi_0, \psi) - 2 \tau \left(
\bar{\nu}(W_{\tau}(\phi_0)) - \bar{\nu}(\psi) \right).
\end{align*}
Now iterate this inequality $n$ times with $\tau = \frac{s}{n}$.  Since
$\bar{\nu}(W_{\frac{s}{n}}^k(\phi_0)) \geq \bar{\nu}(W_{\frac{s}{n}}^n(\phi_0))$
this implies
 \begin{align*}
d^2(W_{\frac{s}{n}}^n(\phi_0), \psi) \leq d^2(\phi_0, \psi) - 2 s \left(
\bar{\nu}(W_{\frac{s}{n}}^n(\phi_0)) - \bar{\nu}(\psi) \right)
\end{align*}
Sending $n \to \infty$ and using the lower semicontinuity of $\bar{\nu}$ yields
the lemma.
\end{proof}
\end{lemma}

\section{The Consistency Theorem} \label{realthmpf}

In this section we prove Theorem \ref{realizationthm}.  To begin we record an
inequality of Chen relating the $K$-energy and distance in $\HH$.

\begin{thm} \label{Kenergydecay} (\cite{ChenSOKM3} Theorem 1.2) Let $\phi_0,
\phi_1 \in \HH$.  Then $\nu(\phi_1) \geq \nu(\phi_0) - d(\phi_0, \phi_1)
\sqrt{\mathcal C(\phi_0)}$.
\end{thm}

\begin{lemma} \label{lowergradsmoothpoints} If $\phi \in \HH$, then $
\brs{\N_{-} \bar{\nu}}(\phi) = \sqrt{\CC(\phi)}$.
\begin{proof} Let $\{\phi_i\} \in \bar{\HH}$ be a sequence in $\bar{\HH}$
converging to $\phi$ in the distance topology.  For each $\phi_i$ choose a
sequence $\{\phi_i^j\} \in \HH$ converging to $\phi_i$ such that
\begin{align*}
 \lim_{j \to \infty} \nu(\phi_i^j) = \bar{\nu}(\phi_i).
\end{align*}
We then compute using Theorem \ref{Kenergydecay},
\begin{align*}
 \frac{\bar{\nu}(\phi) - \bar{\nu}(\phi_i)}{d(\phi, \phi_i)} =&\ \lim_{j \to
\infty} \frac{\bar{\nu}(\phi) - \bar{\nu}(\phi^j_i)}{d(\phi, \phi_i)}\\
\leq&\ \lim_{j \to \infty} \frac{\bar{\nu}(\phi) + \left( d(\phi, \phi_j^i)
\sqrt{\CC(\phi)} - \bar{\nu}(\phi) \right)}{d(\phi, \phi_i)}\\
=&\ \sqrt{\CC(\phi)} \lim_{j \to \infty} \frac{d(\phi,
\phi_j^i)}{d(\phi,\phi_i)}\\
=&\ \sqrt{\CC(\phi)}.
\end{align*}
Since $\phi_i$ was arbitrary, we conclude that $\brs{\N_-
\bar{\nu}} \leq \sqrt{\CC(\phi)}$.  

To show the reverse inequality, we must find an appropriate test curve. 
Unsurprisingly, the right thing to pick is the solution to Calabi flow with
initial condition $\phi = \phi_0$.  As observed in \cite{CalabiExtremal} the
Calabi flow equation is strictly parabolic, and so we have a short-time solution
to Calabi flow $\phi_t$ on $[0, \ge)$.  First, observe that if $\CC(\phi) = 0$
then the desired inequality holds automatically.  Thus assume $\CC(\phi) \neq 0$
and choose $\ge'$ sufficiently small that $\CC(\phi) \neq 0$ for all $t \in
[0,\ge')$.  Observe that for $T \in [0,\ge')$ we have by direct calculation
\begin{align*}
 \nu(\phi_0) - \nu(\phi_T) = \int_0^T \CC(\phi_t) dt.
\end{align*}
Moreover, using $\phi_t$ itself as a test curve in the definition of distance we
obtain
\begin{align*}
 d(\phi_0, \phi_T) \leq&\ \int_0^T \left( \int_M \left( \frac{\del \phi}{\del t}
\right)^2 \gw_{\phi_t}^n \right)^{\frac{1}{2}} dt\\
=&\ \int_0^T \sqrt{\CC(\phi_t)} dt\\
\leq&\ \left( \int_0^T dt \right)^{\frac{1}{2}} \left( \int_0^T \CC(\phi_t) dt
\right)^{\frac{1}{2}}\\
=&\ \sqrt{T} \left( \int_0^T \CC(\phi_t) dt \right)^{\frac{1}{2}}.
\end{align*}
Combining these two statements yields
\begin{align*}
 \lim_{T \to 0} \frac{\bar{\nu}(\phi) - \bar{\nu}(\phi_T)}{d(\phi,\phi_T)} =&\
\lim_{T \to 0} \frac{{\nu}(\phi) - {\nu}(\phi_T)}{d(\phi,\phi_T)}\\
\geq&\ \lim_{T \to 0} \frac{\left( \int_0^T \CC(\phi_t) dt
\right)^{\frac{1}{2}}}{\sqrt{T}}\\
\geq&\ \lim_{T \to 0} \frac{ \left( T \CC(\phi_0) - C T^2
\right)^{\frac{1}{2}}}{\sqrt{T}}\\
=&\ \sqrt{\CC(\phi_0)} = \sqrt{\CC(\phi)}.
\end{align*}
This finishes the proof of the lemma.  Observe that as part the proof we have
shown that if $\phi_t$ denotes a smooth solution to Calabi flow on $[0, T)$,
then for all $t_0 \in [0, T)$ one has
\begin{align} \label{CFder1}
  \lim_{h \to 0} \frac{\bar{\nu}(\phi_{t_0 +
h}) - \bar{\nu}(\phi_{t_0})}{d(\phi_{t_0}, \phi_{t_0 + h})} =
\sqrt{\CC(\phi_{t_0})} = \brs{\N_-
\bar{\nu}}(\phi_{t_0}).
 \end{align}
 Also, this has the further implication that, with the same setup,
 \begin{align} \label{CFder2}
  \lim_{h \to 0} \frac{d(\phi_{t_0}, \phi_{t_0+h})}{h} = \sqrt{\CC(\phi_{t_0})}.
 \end{align}
\end{proof}
\end{lemma}

\begin{proof}[Proof of Theorem \ref{realizationthm}]
Since both $\phi_t$ and $\til{\phi_t}$ are continuous paths in the distance
topology, the function $f(t) = d(\phi_t, \til{\phi}_t)$ is continuous, and $f(0)
= 0$ since $\phi_0 = \til{\phi}_0$.  By standard measure-theoretic lemmas it
suffices to show that 
 \begin{align*}
  D^+ f (t) := \limsup_{h \to 0^+} \frac{f(t + h) - f(t)}{h} \leq 0
 \end{align*}
for all $t \in [0, T)$.  Suppose this were false, and there existed
$\tau \in [0, T)$ and a sequence $t_k \to 0$ such that
 \begin{align} \label{real5}
  \lim_{k \to \infty} \frac{d(\phi_{\tau + t_k}, \til{\phi}_{\tau + t_k}) -
d(\phi_{\tau}, \til{\phi}_{\tau})}{t_k} > 0.
 \end{align}
Let $\psi_t$ denote the $K$-energy minimizing movement with initial condition
$\phi_{\tau}$.  For notational convenience we will shift the time variable so
that $\psi$ is thought to exist on $[\tau, \infty)$, i.e. $\psi_{\tau} =
\phi_{\tau}$.  Let us first rule out a particular case.  Suppose $\brs{\N_-
\bar{\nu}}(\phi_{\tau}) = 0$.  It follows from Theorem \ref{mayer2} and Lemma
\ref{mayer3} that $\psi_{t} = \psi_{\tau}$ for all $t \geq \tau$.  Moreover,
from Lemma \ref{lowergradsmoothpoints} we conclude that $\CC(\phi_{\tau}) = 0$,
and so
$\phi_{\tau}$ is a stationary solution of Calabi flow, i.e. $\phi_{t} =
\phi_{\tau} = \psi_{\tau} = \psi_t$ for all $t \geq \tau$.  It thus follows from
Theorem
\ref{distnonincr} that
\begin{align*}
 d(\phi_{\tau + t_k}, \til{\phi}_{\tau + t_k}) = d(\psi_{\tau}, \til{\phi}_{\tau
+ t_k}) = d(\psi_{\tau + t_k},
\til{\phi}_{\tau + t_k}) \leq d(\psi_{\tau}, \til{\phi}_{\tau}) = d(\phi_{\tau},
\til{\phi}_{\tau}),
\end{align*}
contradicting (\ref{real5}).

Now we assume $\brs{\N_- \bar{\nu}}(\phi_{\tau}) \neq 0$.  The strategy of what
follows is summarized in Figure
\ref{fig1}.  On the one hand the $K$-energy minimizing movement with initial
condition $\phi_{\tau}$, denoted $\psi_{\tau}$, ought to stay within the dotted
line by the distance nonincreasing property of Theorem \ref{distnonincr}.  But
on the other hand we can show that $\psi$ agrees to first order with $\phi$ at
$\tau$, contradicting this fact.
\begin{figure}[ht]
\begin{tikzpicture} [scale = 0.8]

\draw  plot[smooth, tension=.7] coordinates {(-4.2,3.6) (-2.0141,1.0123)
(0.2245,-0.1383) (2.6779,-0.455) (5.4547,0.0196)};
\node (v1) at (-4.2,3.6) {};
\node (v2) at (2.4674,-5.6618) {};
\node (v3) at (-0.7387,0.1926) {};
\node (v4) at (-1.1,-0.6906) {};

\draw [color = red] (v1) -- (v2);

\draw  [fill = black] (v3) circle (0.04);
\draw  [color=red,fill = red] (v4) circle (0.04);
\node at (-0.4458,0.373) {\tiny{$\widetilde{\phi}_{\tau}$}};
\node at (-1.9347,-0.6906) {\tiny{$\phi_{\tau} = \psi_{\tau}$}};
\draw [color = blue] plot[smooth, tension=.7] coordinates {(-1.081,-0.7074)
(-0.2274,-1.8121) (0.8271,-2.7494) (2.3168,-3.4022) (4.2416,-3.7202)};
\node at (3.3328,-3.3177) {\tiny{\textcolor{blue}{$\psi$}}};
\draw [dashed] plot[smooth, tension=.7] coordinates {(-1.081,-0.7074)
(0.3417,-1.1426) (2.7352,-1.3434) (5.5304,-0.858)};
\draw [<->] (1.8828,-1.3552) -- (1.8828,-0.4764);
\node at (2.8021,-0.8413) {\tiny{$d(\phi_{\tau}, \widetilde{\phi}_{\tau})$}};
\node at (4.812,0.1387) {\tiny{$\widetilde{\phi}$}};

\node at (2.2783,-4.9141) {\textcolor{red}{\tiny{$\phi$}}};
\end{tikzpicture}
\caption{A Hypothetical Bifurcation}
\label{fig1}
\end{figure}
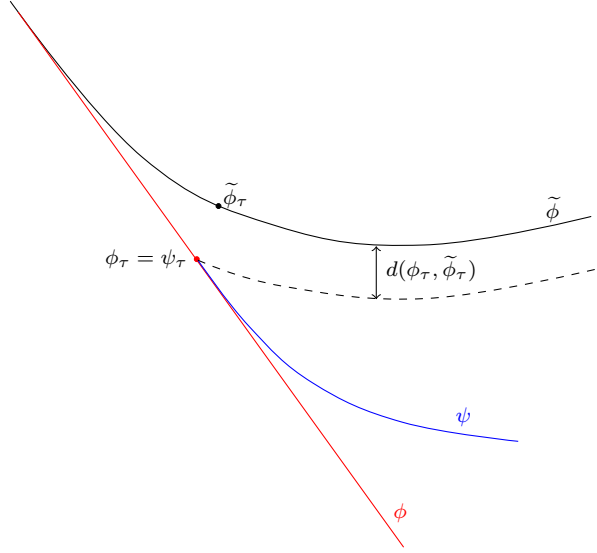
To begin, observe that by (\ref{CFder1}) the sequence of points $ \{
\phi_{\tau + t_k} \}$ satisfies (\ref{ud10}), and therefore by Theorem
\ref{uniquedirection} we conclude that there exists a sequence of times $s_k \to
0$ such that
\begin{align} \label{real7}
 \lim_{k \to \infty} \frac{d(\phi_{\tau + t_k}, \psi_{\tau +
s_k})}{d(\phi_{\tau+t_k},\psi_{\tau})} = 0.
\end{align}
Observe then that by the triangle inequality and Theorem \ref{distnonincr} we
have
\begin{align*}
 \lim_{k \to \infty} \frac{d(\til{\phi}_{\tau + t_k}, {\phi}_{\tau + t_k}) -
d(\til{\phi}_{\tau}, \phi_{\tau})}{t_k} \leq&\ \lim_{k \to \infty}
\frac{d(\til{\phi}_{\tau + t_k}, \psi_{\tau + t_k}) + d(\psi_{\tau + t_k},
{\phi}_{\tau + t_k}) - d(\til{\phi}_{\tau}, \psi_{\tau})}{t_k}\\
\leq&\ \lim_{k \to \infty} \frac{d(\psi_{\tau + t_k}, {\phi}_{\tau +
t_k})}{t_k}\\
\leq&\ \lim_{k \to \infty} \frac{d(\psi_{\tau + t_k}, \psi_{\tau + s_k}) +
d(\psi_{\tau + s_k}, \phi_{\tau + t_k})}{t_k}\\
=&\ I + II.
\end{align*}
We claim that both $I$ and $II$ are zero, contradicting the hypothesis
(\ref{real5}) and
finishing the proof.  For the term $I$ we recall from Remark
\ref{uniquedirectionrmk} that the times $s_k$ are chosen
so that
\begin{align*}
 d(\phi_{\tau}, \psi_{\tau + s_k}) = d(\phi_{\tau}, \phi_{\tau + t_k}) =:d_k.
\end{align*}
Note also that by (\ref{CFder2}) we have that
\begin{align*}
 d_k = d(\phi_{\tau}, \phi_{\tau + t_k}) = t_k \brs{\N_- \bar{\nu}}(\phi_{\tau})
+ t_k o(t_k).
\end{align*}
Also, by Theorem \ref{mayer2} we conclude that
\begin{align*}
 d_k = d(\phi_{\tau}, \psi_{\tau + s_k}) = d(\psi_{\tau}, \psi_{\tau + s_k}) =
s_k \brs{\N_- \bar{\nu}}(\psi_{\tau}) + s_k o(s_k).
\end{align*}
By combining the two equations above and using that $\brs{\N_-
\bar{\nu}}(\phi_{\tau}) = \brs{\N_- \bar{\nu}}(\psi_{\tau}) \neq 0$, a number of
consequences follow.  In particular we conclude that there exist constants $C,
k_0 > 0$ such that for all $k \geq k_0$ one has
\begin{align*}
 s_k \leq C t_k.
\end{align*}
We assume that all values of $k$ below satisfy $k \geq k_0$.  In particular,
using the three equations displayed above we conclude that
\begin{align*}
\brs{ t_k - s_k} = t_k o(t_k) + s_k o(s_k) = t_k o(t_k).
\end{align*}
Now assume that for some given $k$ we have $t_k >
s_k$.  By Theorem \ref{mayer2} we conclude that there is a constant $C > 0$ so
that
\begin{align*}
 d(\psi_{\tau}, \psi_{\tau + (t_k - s_k)}) \leq C (t_k - s_k) \leq C t_k o(t_k).
\end{align*}
By Theorem \ref{distnonincr} we obtain
\begin{align*}
 d(\psi_{\tau + s_k}, \psi_{\tau + t_k)}) \leq d(\psi_{\tau}, \psi_{\tau + (t_k
- s_k)}) \leq C t_k o(t_k).
\end{align*}
The case $s_k > t_k$ yields a similar inequality.  Thus finally we conclude that
\begin{align*}
 I =&\ \lim_{k \to \infty} \frac{d(\psi_{\tau + s_k}, \psi_{\tau + t_k})}{t_k}
\leq \lim_{k \to \infty} \frac{t_k o(t_k)}{t_k}  = 0.
\end{align*}
For the term $II$ we note using (\ref{CFder2}) and (\ref{real7}) that
\begin{align*}
 \lim_{k \to \infty} \frac{d(\psi_{\tau + s_k}, \phi_{\tau + t_k})}{t_k} =&\
\lim_{k \to \infty} \frac{d(\phi_{\tau+t_k}, \phi_{\tau})}{t_k} \cdot
\frac{d(\psi_{\tau + s_k}, \phi_{\tau + t_k})}{d(\phi_{\tau + t_k},
\phi_{\tau})}\\
 =&\ \left( \lim_{k \to \infty} \frac{d(\phi_{\tau+t_k}, \phi_{\tau})}{t_k}
\right) \cdot \left( \lim_{k \to \infty} \frac{d(\phi_{\tau +
t_k},\psi_{\tau + s_k})}{d(\phi_{\tau + t_k}, \psi_{\tau})} \right)\\
 =&\ \brs{\N_- \bar{\nu}}(\phi_{\tau}) \cdot 0\\
 =&\ 0.
\end{align*}
\end{proof}

\noindent Next we give the proof of Corollaries \ref{distancecor},
\ref{distancecor2}, \ref{minimizingcor}, and \ref{minimizingcor2}.

\begin{proof}[Proof of Corollary \ref{distancecor}] 
This follows directly from \cite{Mayer} Theorem 2.2.  A precise estimate of the
constant $B$ as referenced in that statement is given in \cite{SMM} Lemma 5.8,
and has the claimed dependencies.
\end{proof}

\begin{proof}[Proof of Corollary \ref{distancecor2}] 
 Again we want to apply \cite{Mayer} Theorem 2.2.  One observes that the proof
is in a sense a formal application of the convexity property and the estimate
of \cite{Mayer} Lemma 1.11.  This lemma claims that given $\phi_0 \in
\bar{\HH}$, $T > 0$, there is a constant $B$ with complicated dependencies such
that for all $\tau > 0$ and $j \in \mathbb N$ such that $j \tau \leq T$, one has
\begin{align*}
 d^2 (\phi_0, W_{\tau}^j(\phi_0)) \leq B j \tau.
\end{align*}
Therefore it suffices to prove this estimate with weaker dependencies in this
setting.  So, by the definition of the resolvent operator, we have
\begin{align*}
\frac{d^2(W_{\tau}^{j+1}(\phi_0), W_{\tau}^j(\phi_0)}{2 \tau} \leq
\bar{\nu}(W_{\tau}^j(\phi_0)) - \bar{\nu}(W_{\tau}^{j+1}(\phi_0)).
\end{align*}
Therefore by the triangle inequality and the Cauchy-Schwarz inequality we obtain
\begin{align*}
 d^2(\phi_0, W_{\tau}^j(\phi_0)) \leq&\ \left( \sum_{i = 0}^{j-1}
d(W_{\tau}^i(\phi_0), W_{\tau}^{i+1}(\phi_0)) \right)^2\\
 \leq&\ j \sum_{i=0}^{j-1}d^2(W_{\tau}^i(\phi_0), W_{\tau}^{i+1}(\phi_0)).
\end{align*}
Combining these inequalities and using that $\bar{\nu}$ is bounded below yields
\begin{align*}
 d^2( \phi_0, W_{\tau}^j(\phi_0)) \leq 2 j \tau \left(\bar{\nu}(\phi_0) -
\bar{\nu}(W_{\tau}^j(\phi_0)) \right) \leq C(\bar{\nu}(\phi_0), [\gw]) j \tau.
\end{align*}
Thus we have obtained the required bound depending only on $\bar{\nu}(\phi_0)$
and the underlying K\"ahler class $[\gw]$.  The corollary follows.
\end{proof}

\begin{proof} [Proof of Corollary \ref{minimizingcor}] The first statement
follows directly from \cite{Mayer} Theorem 2.39 and the fact that
\begin{align*}
 \inf_{\bar{\phi} \in \bar{\HH}} \bar{\nu}(\bar{\phi}) = \inf_{\phi \in \HH}
\nu(\phi).
\end{align*}
Thus the second statement follows directly from the first and Theorem
\ref{realizationthm}. 
\end{proof}

\begin{proof}[Proof of Corollary \ref{minimizingcor2}] 
Suppose the claim is false, i.e. for the given Calabi flow
$\phi_t$,
\begin{align*}
\gb := \lim_{t \to \infty} \CC(\phi_t) > \inf_{\phi \in \HH} \CC(\phi).
\end{align*}
Choose $\psi \in \HH$ such that $\CC(\psi) = \ga < \gb$, and let $\psi_t$ denote
the KEMM with initial condition $\psi$ guaranteed by Theorem \ref{MMex}.  Note
that by Lemmas \ref{mayer4} and \ref{lowergradsmoothpoints} it follows
that for all $t \geq 0$ one has
\begin{align*}
\brs{\N_- \bar{\nu}}(\psi_t) \leq \brs{\N_- \bar{\nu}}(\psi_0) =
\CC(\psi) = \ga.
\end{align*}
Also, it follows from \cite{Mayer} Corollary 2.18 that for a general KEMM $\psi_t$ one
has
\begin{align} \label{KEMMgradientprop}
\bar{\nu}(\psi_0) - \bar{\nu}(\psi_t) = \int_0^t \brs{\N_- \bar{\nu}}^2(\psi_s)
ds \leq t \ga.
\end{align}
Also, applying the gradient flow property to the Calabi flow solution $\phi_t$ yields
\begin{align*}
\bar{\nu}(\phi_0) - \bar{\nu}(\phi_t) \geq&\ \gb t.
\end{align*}
Combining these inequalities yields that, for a constant $C$ depending on $i$
and all $t \geq 0$, one has
\begin{align} \label{CC7}
\bar{\nu}(\psi_t) \geq&\ \bar{\nu}(\phi_t) + \left( \gb - \ga \right) t - C.
\end{align}
Applying Lemma \ref{minimizinglemma} and Theorem \ref{distnonincr} we obtain
\begin{gather} \label{CC10}
 \begin{split}
 0 \leq&\ d^2(\psi_{t+s}, \phi_t)\\
 \leq&\ d^2(\psi_t, \phi_t) - 2s \left( \bar{\nu}(\psi_{t+s}) -
\bar{\nu}(\phi_t) \right)\\
 \leq&\ d^2(\psi_0, \phi_0) - 2s \left( \bar{\nu}(\psi_t) -
\bar{\nu}(\phi_t) \right)  + 2s \left(\bar{\nu}(\psi_t) -
\bar{\nu}(\psi_{t+s})\right).  
 \end{split}
\end{gather}
Now choose $s = 1$ and observe by (\ref{KEMMgradientprop}) that
\begin{align} \label{CC20}
\bar{\nu}(\psi_t) - \bar{\nu}(\psi_{t+1}) = \int_t^{t+1} \brs{\N_-
\bar{\nu}}^2(\psi_s) ds \leq&\ C.
\end{align}
Thus plugging in (\ref{CC7}) and (\ref{CC20}) into (\ref{CC10}) yields
\begin{align*}
 0 \leq&\ C - 2 (\gb - \ga) t,
\end{align*}
which is a contradiction for sufficiently large $t$.  The corollary follows.
\end{proof}

\section{Convergence results} \label{convsec}

In this section we prove Theorem \ref{convthm}.  To begin we define a notion of
weak convergence in NPC spaces first introduced
by Jost \cite{Jost}.

\begin{defn} Let $(X, d)$ be a complete NPC space.  Given $\{x_n\} \subset X$ a
bounded sequence and a point $x \in X$, the \emph{asymptotic radius of $\{x_n\}$
around $x$} is
\begin{align*}
r(\{x_n\}, x) = \limsup_{n \to \infty} d(x_n, x).
\end{align*}
Moreover, the \emph{asymptotic radius of $\{x_n\}$} is
\begin{align*}
r(\{x_n\}) = \inf_{x \in X} r(\{x_n\}, x).
\end{align*}
Also, we declare that $x \in X$ is the \emph{asymptotic center of $\{x_n\}$} if
$r(\{x_n\},x) = r(\{x_n\})$.  We will show in Lemma \ref{uniqueac} that the
asymptotic center always exists and is unique.  Finally, we say that
\emph{$\{x_n\}$ weakly converges to $x$ } if $x$ is the asymptotic center of
every subsequence of $\{x_n\}$.  We say that $z \in X$ is a \emph{weak cluster
point} of $\{x_n\}$ if there is a subsequence which converges weakly to $z$.
\end{defn}

\begin{rmk} The definition of weak convergence above is a generalization of the
definition of weak
convergence in Hilbert spaces.
\end{rmk}

\begin{lemma} \label{uniqueac} Given $(X, d)$ a complete NPC space and $\{x_n\}$
a bounded sequence in $X$, there
exists a unique asymptotic center for $\{x_n\}$.
\begin{proof} First we show uniqueness.  Suppose $x$ and $y$ are both asymptotic
centers of $\{x_n\}$.  By hypothesis we have
\begin{align*}
\limsup_{n \to \infty} d(x_{n}, x) = r(\{x_n\}, x) = r(\{x_n\}) = r(\{x_n\}, y)
= \limsup_{n \to \infty} d( x_n, y).
\end{align*}
Consider the geodesic $\gg : [0,1] \to X$ connecting $x$ to $y$.  Fix any $t \in
(0,1)$.  By applying the triangle inequality we conclude for any $n$,
\begin{align*}
d^2(x_{n}, \gg(t)) \leq&\ (1-t) d^2(x_{n}, x) + t d^2(x_{n}, y) - t(1-t)
d^2(x,y)\\
\leq&\ (1-t) r^2(\{x_n\}, x) + t r^2(\{x_n\}, y) - t(1-t) d^2(x,y)\\
=&\ r^2(\{x_n\}) - t(1-t) d^2(x,y).
\end{align*}
Taking the limsup as $n$ goes to infinity we obtain
\begin{align*}
r(\{x_n\}) =&\ \inf_{x \in X} r(\{x_n\}, x)\\
\leq&\ r(\{x_n\}, \gg(t))\\
\leq&\ \sqrt{ r^2(\{x_n\} - t(1-t) d(x,y)}.
\end{align*}
It follows that $d(x,y) = 0$, and so $x = y$ as required.

Now we show existence.  As the sequence is bounded, certainly there exists $x
\in X$ such that $r(\{x_n\}, x) < \infty$.  Thus choose a sequence $\{y_n\}$
realizing $\inf_{x \in X} r(\{x_n\}, x)$.  By repeating the argument estimate
above for uniqueness shows directly that $\{y_n\}$ is a Cauchy sequence.  Since
$X$ is complete there exists a limit point $y_{\infty}$ which is an asymptotic
center.
\end{proof}
\end{lemma}

\begin{lemma} (\cite{BH} Proposition 2.4 pg. 176) \label{projectionlemma} Let
$(X, d)$ be a complete NPC space.  Let
$C$ denote a complete convex subset of $X$.  Then
\begin{enumerate}
\item{ For every $x \in X$, there exists a unique point $\pi_C(x) \in C$ such
that $d(x, \pi_C(x)) = d(x, C)$.}
\item{ If $\gg$ denotes the geodesic connecting $x$ to $\pi_C(x)$ and $y \in
\gg$ then $\pi_C (x) = \pi_C(y)$}
\item{ If $x \in X \backslash C$ and $y \in C$ satisfies $\pi_C(x) \neq y$, then
$\ga(x,\pi_C(x),y) \geq \frac{\pi}{2}$.}
\end{enumerate}
\begin{proof} (1) Consider a sequence of points $\{y_n\} \in C$ realizing $d(x,
C)$.  We will show that $\{y_n\}$ is a Cauchy sequence, which simultaneously
establishes existence and uniqueness.  

(2) This follows directly from the triangle inequality.

(3) First note that if $\ga(x,\pi_C(x),y) < \frac{\pi}{2}$, then by choosing
$x'$ on the geodesic connecting $x$ to $\pi_C(x)$ very close to $\pi_C(x)$ and
likewise choosing $y'$ on the geodesic connecting $y$ to $\pi_C(x)$ sufficiently
close to $\pi_C(x)$ , we can guarantee that the corresponding angle in the
comparison triangle $
\bar{\gD}(x',\pi_C(x), y') \subset \mathbb R^2$ is also less than
$\frac{\pi}{2}$.  But using the NPC condition this implies that there is a point
$p \in C$ on the geodesic connecting $\pi_C(x)$ to $y$ satisfying $d(x',p) <
d(x',\pi_C(x))$.  But by (2) we have $d(x',\pi_C(x)) = d(x',C)$, a
contradiction.
\end{proof}
\end{lemma}

\begin{lemma} (cf. \cite{Espinola} Proposition 5.2) \label{wkconvequiv} If a
bounded sequence $\{x_n\} \subset X$
converges weakly to a point $x \in X$ then for any geodesic $\gg$ passing
through $x$, one has
\begin{align*}
 \lim_{n \to \infty} d(x, \pi_{\gg}(x_n)) = 0.
\end{align*}
\begin{proof} If the claim were false then let $\gg$ denote a geodesic segment
containing $x$ such that
\begin{align*}
\lim_{n \to \infty} d(x, \pi_{\gg}(x_n)) \neq 0
\end{align*}
Observe that since the segment $\gg$ is compact there exists a subsequence
$\{x_{n_i}\}$ and a point $y \in \gg, y \neq x$ such that $\{\pi_{\gg} x_{n_i}\}
\to y$. 
By the definition of the projection operator, for all $n_i$ one has $d(x_{n_i},
\pi_{\gg}(x_{n_i})) \leq d(x_{n_i}, x)$.  Taking the limit as $i \to \infty$
yields
\begin{align*}
\lim_{i \to \infty} d(x_{n_i}, y) \leq \lim_{i \to \infty} d(x_{n_i}, x).
\end{align*}
Since $x$ is the asymptotic center of every subsequence of $\{x_n\}$, it follows
from the above inequality that $y$ is the asymptotic center of the sequence
$\{x_{n_i}\}$.  However, by Lemma \ref{uniqueac} asymptotic centers are unique
and so $y = x$, a contradiction.
\end{proof}
\end{lemma}

\begin{lemma} (\cite{Bacak2} Lemma 3.1) \label{wklimitsinconv} Let $C$ denote a
closed convex subset of a
complete NPC space $(X, d)$.  If $\{x_n\} \subset C$ and $\{x_n\}$ converges
weakly to $x$, then $x \in C$.
 \begin{proof}
Suppose $x \notin C$.  Let $\gg : [0,1] \to X$ denote the geodesic connecting
$x$ to $\pi_C(x)$.  We aim to show that $\pi_{\gg}(x_n) = \pi_C(x)$ for all $n$.

We argue by contradiction and assume $\pi_{\gg}(x_n) \neq \pi_C(x)$.  Observe
that if $\pi_{\gg}(x_n) \in C$, then by Lemma \ref{projectionlemma} (2) we
would have $\pi_{\gg}(x_n) = \pi_C \pi_{\gg}(x_n) = \pi_C x$, thus
$\pi_{\gg}(x_n) \notin C$.
Applying Lemma \ref{projectionlemma} (3) we conclude
\begin{align*}
 \ga(\pi_{\gg}(x_n), \pi_C(x), x_n) = \ga(\pi_{\gg}(x_n), \pi_C \pi_{\gg}(x_n),
x_n) \geq \frac{\pi}{2}.
\end{align*}
On the other hand $\gg$ is itself a closed convex set.  Moreover note that $x_n
\notin \gg$ for otherwise $ \pi_{\gg}(x_n) = x_n \in C$, contradicting the
argument above that $\pi_{\gg}(x_n) \notin C$.  Thus applying Lemma
\ref{projectionlemma} (3) to the set $\gg$ we conclude
\begin{align*}
 \ga(x_n, \pi_{\gg}(x_n), \pi_C(x)) \geq \frac{\pi}{2}.
\end{align*}
Since $\ga(\pi_{\gg}(x_n), x_n, \pi_{C}(x)) > 0$, these three inequalities
contradict the triangle inequality for NPC space, finishing the proof of the
claim that $\pi_{\gg}(x_n) = \pi_C(x)$.  It follows that
\begin{align*}
 \lim_{n \to \infty} d(\pi_{\gg}(x_n), x) = \lim_{n \to \infty} d(\pi_C(x), x))
> 0.
\end{align*}
But since $\{x_n\}$ converges weakly to $x$, Lemma \ref{wkconvequiv} guarantees
that for the geodesic $\gg$,
\begin{align*}
 \lim_{n \to \infty} d(\pi_{\gg}(x_n), x) = 0.
\end{align*}
This is a contradiction, and so $x \in C$.
\end{proof}
\end{lemma}

\begin{lemma} \label{weaklsclemma} (\cite{Bacak} Lemma 3.1) Let $(X, d)$ be a
complete NPC space.  If $f
: X \to (-\infty,
\infty]$ is a lower semicontinuous convex function, then it is weakly lower
semicontinuous.
\begin{proof} If the claim were false, we could find $x \in X$ and $\{x_n\} \in
X$ converging weakly to $x$ such that
\begin{align*}
 \liminf_{n \to \infty} f(x_n) < f(x).
\end{align*}
In particular, there is a subsequence $x_{n_k}$ and $\ge > 0$ such that
$f(x_{n_k}) < f(x) - \ge$ for all $k$.  Let $C$ denote the closure of the convex
hull of $\{x_{n_k}\}$.  From convexity and lower semicontinuity of $f$ we
conclude that for all $y \in C$ one has
\begin{align*}
 f(y) < f(x) - \ge.
\end{align*}
However, by Lemma \ref{wklimitsinconv} we conclude that $x \in C$, and thus we
obtain
\begin{align*}
 f(x) < f(x) - \ge,
\end{align*}
a contradiction. 
\end{proof}
\end{lemma}

Next we record the convergence theorem for minimizing movements proved by
Ba\u{c}\'ak
mentioned in the introduction. 

\begin{thm} \label{Bacakthm} (\cite{Bacak} Theorem 1.5) Given $(X, d)$ a
complete NPC space and $f : X \to (-\infty, \infty]$ a lower semicontinuous
convex function.  Assume that $f$ attains its minimum on $X$.  Then for all $x
\in X$, the $f$-minimizing movement with initial condition $x$ converges weakly
to a minimizer of $f$ as $\gl \to \infty$.
\begin{proof} Fix a sequence $\{t_n\} \to \infty$, and for notational
convenience let $x_n = F_{t_n} x$.  Let
\begin{align*}
C = \{y \in X | f(y) = \inf_{x \in X} f(x) \}.
\end{align*}
By assumption $C \neq \emptyset$.  Note that it is clear that given any $y \in
C$, the minimizing movement with initial condition $y$ is stationary, i.e. $F_t
y = y$ for all $t \geq 0$.  It then follows from the distance nonincreasing
property of $f$-minimizing movements that for
any $x \in X$, $y \in C$ and $n > m$ one has $d(x_{n}, y) \leq d(x_m, y)$.

Now we claim that if all weak cluster points of $\{x_n\}$ lie in $C$ then
there is a unique weak cluster point of $\{x_n\}$ in $C$.  Suppose $c_1, c_2 \in
C$ are weak cluster points of $\{x_n\}$.  In particular, there exists a
subsequence $\{x_{n_k} \}$ converging weakly to $c_1$ and a subsequence
$\{x_{m_l} \}$ converging weakly to $c_2$.  Without loss of generality let
us assume that $r(\{x_{n_k} \}) \leq r(\{x_{m_l}\})$.  Fix $\ge > 0$,
and then fix $K \in \mathbb N$ such that $d(x_{n_k}, c_1) < r(\{x_{n_k}\}) +
\ge$ for all $k \geq K$.  By the distance nonincreasing property we immediately
conclude that
\begin{align*}
 d(x_{m_k}, c_1) < r(\{x_{n_k}\}) + \ge \leq&\ r(\{x_{n_l} \}) + \ge
\end{align*}
for all $l$ large enough to ensure $n_l \geq n_K$.  It follows that $c_1$ is an
asymptotic center for $\{x_{m_k}\}$, but these are unique by Lemma
\ref{uniqueac}, and
so $c_1 = c_2$.

We now show that all weak cluster points of $\{x_n\}$ do indeed lie in $C$. 
With the claim of uniqueness above, the proof will be finished.  By
(\cite{Mayer} Theorem 2.39) the sequence $\{x_n\}$ is minimizing for $f$, i.e.
\begin{align*}
\lim_{n \to \infty} f(x_n) = \inf_{x \in X} f(x). 
\end{align*}
Since $f$ is weakly lower semicontinuous by Lemma \ref{weaklsclemma}, we
conclude
that all weak cluster points of $\{x_n\}$ are in $C$.
\end{proof}
\end{thm}

\begin{proof}[Proof of Theorem \ref{convthm}] In \cite{SMM} Lemma 5.9 it was
established that $(\bar{\HH}, d)$ is an NPC space, and in \cite{SMM} \S 5 we
established that $\bar{\nu}$ is a lower semicontinuous convex function.  The
theorem follows directly from Theorem \ref{Bacakthm}.
\end{proof}

\begin{proof}[Proof of Corollary \ref{convcor}] It follows from \cite{ChenTian}
Theorem 1.1.2 that $\phi_{\infty}$
is a minimizer for $\nu$, as defined on $\HH$.  From the definition of
$\bar{\nu}$ it follows immediately that $\phi_{\infty}$ is a minimizer for
$\bar{\nu}$.  The corollary follows from Theorem
\ref{convthm}.
\end{proof}

\bibliographystyle{hamsplain}

\end{document}